\documentclass[11pt]{article}
\usepackage{amssymb,amsthm,amsmath,algorithm2e,float,cite,enumerate}
\usepackage{geometry}

\newtheorem{theorem}{Theorem}[section]
\newtheorem{lemma}[theorem]{Lemma}
\newtheorem{corollary}[theorem]{Corollary}
\newtheorem{conjecture}[theorem]{Conjecture}
\newtheorem{claim}{Claim}

\usepackage{setspace}

\begin{document}

\onehalfspace

\title{Partial immunization of trees}

\author{Mitre C. Dourado$^1$ 
\and Stefan Ehard$^2$
\and Lucia D. Penso$^2$
\and Dieter Rautenbach$^2$}

\date{}

\maketitle

\begin{center}
$^1$ Instituto de Matem\'{a}tica\\
Universidade Federal do Rio de Janeiro, Rio de Janeiro, Brazil,
\texttt{mitre@dcc.ufrj.br}\\[3mm]
$^2$ Institut f\"{u}r Optimierung und Operations Research, 
Universit\"{a}t Ulm, Ulm, Germany,
\{\texttt{stefan.ehard,lucia.penso,dieter.rautenbach}\}\texttt{@uni-ulm.de}\\[3mm]
\end{center}

\begin{abstract}
For a graph $G$ 
and an integer-valued function $\tau$ on its vertex set,
a dynamic monopoly is a set of vertices of $G$
such that iteratively adding to it vertices $u$ of $G$ 
that have at least $\tau(u)$ neighbors in it
eventually yields the vertex set of $G$.
We study the problem of maximizing the minimum order of a dynamic monopoly
by increasing the threshold values of individual vertices
subject to vertex-dependent lower and upper bounds,
and fixing the total increase.
We solve this problem efficiently for trees,
which extends a result of Khoshkhah and Zaker
(On the largest dynamic monopolies of graphs with a given average threshold, 
Canadian Mathematical Bulletin 58 (2015) 306-316).
\end{abstract}

{\small 
\begin{tabular}{lp{13cm}}
\end{tabular}
}


\section{Introduction}

As a simple model for an infection process within a network \cite{keklta,drro,dori}
one can consider a graph $G$
in which each vertex $u$ is assigned a non-negative integral threshold value $\tau(u)$ 
quantifying how many infected neighbors of $u$ 
are required to spread the infection to $u$.
In this setting, 
a dynamic monopoly of $(G,\tau)$ is a set $D$ of vertices 
such that an infection starting in $D$ 
spreads to all of $G$,
and the smallest order ${\rm dyn}(G,\tau)$ of such a dynamic monopoly 
measures the vulnerability of $G$ for the given threshold values.

Khoshkhah and Zaker \cite{khza} 
consider the maximum of ${\rm dyn}(G,\tau)$
over all choices for the function $\tau$
such that the average threshold is at most 
some positive real $\bar{\tau}$.
They show that this maximum equals
\begin{eqnarray}\label{ez}
\max\left\{ k:\sum\limits_{i=1}^k(d_G(u_i)+1)\leq n(G)\bar{\tau}\right\},
\end{eqnarray}
where $u_1,\ldots,u_{n(G)}$ is a linear ordering of the vertices of $G$ 
with non-decreasing vertex degrees $d_G(u_1)\leq \ldots \leq d_G(u_{n(G)})$.
To obtain this simple formula
one has to allow $d_G(u)+1$ as a threshold value for vertices $u$,
a value that makes these vertices completely immune to the infection,
and forces every dynamic monopoly to contain them.
Requiring $\tau(u)\leq d_G(u)$ for every vertex $u$ of $G$ leads to a harder problem;
Khoshkhah and Zaker \cite{khza} show hardness for planar graphs
and describe an efficient algorithm for trees.
In the present paper we consider their problem 
with additional vertex-dependent lower and upper bounds 
on the threshold values. 
As our main result, we describe an efficient algorithm for trees 
based on a completely different approach than the one in \cite{khza}.

In order to phrase the problem and our results exactly,
and to discuss further related work, 
we introduce some terminology.
Let $G$ be a finite, simple, and undirected graph.
A {\it threshold function} for $G$ is a function 
from the vertex set $V(G)$ of $G$ to the set of integers.
For notational simplicity, we allow negative threshold values.
Let $\tau\in \mathbb{Z}^{V(G)}$ be a threshold function for $G$.
For a set $D$ of vertices of $G$,
the {\it hull $H_{(G,\tau)}(D)$ of $D$ in $(G,\tau)$}
is the smallest set $H$ of vertices of $G$ such that 
$D\subseteq H$, and 
$u\in H$ for every vertex $u$ of $G$ with $|H\cap N_G(u)|\geq \tau (u)$.
Clearly, the set $H_{(G,\tau)}(D)$ is obtained 
by starting with $D$, and 
iteratively adding vertices $u$ 
that have at least $\tau(u)$ neighbors in the current set
as long as possible.
With this notation, 
the set $D$ is a {\it dynamic monopoly of $(G,\tau)$}
if $H_{(G,\tau)}(D)$ equals the vertex set of $G$,
and ${\rm dyn}(G,\tau)$ is the minimum order of such a set.
A dynamic monopoly of $(G,\tau)$ of order ${\rm dyn}(G,\tau)$ is {\it minimum}.
The parameter ${\rm dyn}(G,\tau)$ is computationally hard \cite{ch,cedoperasz};
next to general bounds \cite{acbewo,gera,chly} 
efficient algorithms are only known for essentially tree-structured instances  
\cite{cicogamipeva,ch,cedoperasz,chhuliwuye,behelone}.

We can now phrase the problem we consider:
For a given graph $G$,
two functions $\tau,\iota_{\max}\in \mathbb{Z}^{V(G)}$, 
and a non-negative integer {\it budget} $b$, let 
${\rm vacc}(G,\tau,\iota_{\max},b)$ be defined as
\begin{eqnarray}
\max\Big\{ {\rm dyn}(G,\tau+\iota):
\iota\in \mathbb{Z}^{V(G)}, 
0\leq \iota\leq \iota_{\max},\mbox{ and }\iota(V(G))=b\Big\},\label{ed1}
\end{eqnarray}
where 
inequalities between functions are meant pointwise,
and $\iota(V(G))=\sum\limits_{u\in V(G)}\iota(u)$.
The function $\iota$ is the {\it increment} 
of the original threshold function $\tau$.
The final threshold function $\tau+\iota$ must lie between $\tau$ and $\tau+\iota_{\max}$,
which allows to incorporate vertex-dependent lower and upper bounds.
Note that no such increment $\iota$ exists if $\iota_{\max}(V(G))$ is strictly less than $b$, 
in which case ${\rm vacc}(G,\tau,\iota_{\max},b)$ 
equals $\max\emptyset=-\infty$.
Note that we require $\iota(V(G))=b$ in (\ref{ed1}),
which determines the average final threshold as 
$(\tau(V(G))+b)/n(G)$.
Since 
${\rm dyn}(G,\rho)\leq {\rm dyn}(G,\rho')$
for every two threshold functions $\rho$ and $\rho'$ for $G$ 
with $\rho\leq \rho'$,
for $\iota_{\max}(V(G))\geq b$,
the value in (\ref{ed1}) remains the same when replacing
`$\iota(V(G))=b$' with `$\iota(V(G))\leq b$'
provided that $b\leq \iota_{\max}(V(G))$.

The results of Khoshkhah and Zaker \cite{khza} 
mentioned above can be phrased by saying  
\begin{enumerate}[(i)]
\item that ${\rm vacc}(G,0,d_G+1, n(G)\bar{\tau})$ equals (\ref{ez})
whenever $n(G)\bar{\tau}$ is a non-negative integer at most $\sum\limits_{u\in V(G)}(d_G(u)+1)=2m(G)+n(G)$,
where $m(G)$ is the size of $G$, and
\item that ${\rm vacc}(T,0,d_T,b)$ can be determined 
efficiently whenever $T$ is a tree.
\end{enumerate}
Our main result is the following.

\begin{theorem}\label{theorem1}
For a given tuple $(T,\tau,\iota_{\max},b)$,
where $T$ is a tree of order $n$,
$\tau,\iota_{\max}\in\mathbb{Z}^{V(G)}$,
and $b$ is an integer with $0\leq b\leq \iota_{\max}(V(T))$,
the value ${\rm vacc}(T,\tau,\iota_{\max},b)$
as well as an increment $\iota\in\mathbb{Z}^{V(G)}$ 
with $0\leq \iota\leq \iota_{\max}$ and $\iota(V(G))=b$
such that ${\rm vacc}(T,\tau,\iota_{\max},b)={\rm dyn}\left(T,\tau+\iota\right)$
can be determined in time $O\left(n^2(b+1)^2\right)$.
\end{theorem}
While our approach relies on dynamic programming,
Khoshkhah and Zaker show (ii) using the following result
in combination with a minimum cost flow algorithm.

\begin{theorem}[Khoshkhah and Zaker \cite{khza}]\label{theoremkhza}
For a given tree $T$, and a given integer $b$ with $0\leq b\leq 2m(T)$,
there is a matching $M$ of $T$ such that 
${\rm vacc}(T,0,d_T,b)={\rm dyn}(G,\tau_M)$
and
$\tau_M(V(T))\leq b$,
where
$$
\tau_M:V(T)\to\mathbb{Z}:u\mapsto 
\begin{cases}
d_T(u) &,\mbox{ $u$ is incident with a vertex in $M$, and}\\
0 &,\mbox{ otherwise.}
\end{cases}
$$
\end{theorem}
We believe that the threshold function $\tau_M$ 
considered in Theorem \ref{theoremkhza} 
is a good choice in general, and pose the following.

\begin{conjecture}\label{conjecture1}
For a given graph $G$, and a given integer $b$ with $0\leq b\leq 2m(G)$,
there is a matching $M$ of $G$ such that 
${\rm vacc}(G,0,d_G,b)\leq 2{\rm dyn}(G,\tau_M)$
and
$\tau_M(V(G))\leq b$,
where $\tau_M$ is as in Theorem \ref{theoremkhza}
(with $T$ replaced by $G$).
\end{conjecture}
As a second result we show Conjecture \ref{conjecture1} for some regular graphs.

\begin{theorem}\label{theorem2}
Conjecture \ref{conjecture1} holds if $G$ is $r$-regular and $b\geq (2r-1)(r+1)$.
\end{theorem}
Before we proceed to the proofs of Theorems \ref{theorem1} and \ref{theorem2},
we mention some further related work.
Centeno and Rautenbach \cite{cera} establish bounds 
for the problems considered in \cite{khza}.
In \cite{ehra}, Ehard and Rautenbach consider the following two variants of (\ref{ed1}) for a given triple $(G,\tau,b)$,
where $G$ is a graph, $\tau$ is a threshold function for $G$,
and $b$ is a non-negative integer:
$$
\max\left\{ {\rm dyn}(G-X,\tau):X\in {V(G)\choose b}\right\}
\,\,\,\,\,\,\,\,\,\,\,\mbox{ and }\,\,\,\,\,\,\,\,\,\,\,
\max\left\{ {\rm dyn}(G,\tau_X):X\in {V(G)\choose b}\right\},$$
where
$$\tau_X(u) =
\begin{cases}
d_G(u)+1 & \mbox{, if $u\in X$,}\\
\tau(u) & \mbox{, if $u\in V(G)\setminus X$,}
\end{cases},$$
and ${V(G) \choose b}$ denotes the set of all $b$-element subsets of $V(G)$.
For both variants, they describe efficient algorithms for trees.
In \cite{bhkllerosh} Bhawalkar et al.
study so-called anchored $k$-cores.
For a given graph $G$, and a positive integer $k$, 
the {\it $k$-core} of $G$ is 
the largest induced subgraph of $G$ of minimum degree at least $k$.
It is easy to see that the vertex set of the $k$-core of $G$
equals $V(G)\setminus H_{(G,\tau)}(\emptyset)$
for the special threshold function $\tau=d_G-k+1$.
Now, the {\it anchored $k$-core problem} \cite{bhkllerosh} is to determine 
\begin{eqnarray}\label{ed3}
\max\left\{
\Big|V(G)\setminus H_{(G,\tau_X)}(\emptyset)\Big|:X\in {V(G)\choose b}\right\},
\end{eqnarray}
for a given graph $G$ and non-negative integer $b$.
Bhawalkar et al. show that (\ref{ed3}) is hard to approximate in general,
but can be determined efficiently for $k=2$, and for graphs of bounded treewidth.
Vaccination problems in random settings were studied in \cite{keklta,brjama,de}.

\section{Proofs of Theorem \ref{theorem1} and Theorem \ref{theorem2}}
Throughout this section, 
let $T$ be a tree rooted in some vertex $r$, 
and let $\tau,\iota_{\max}\in \mathbb{Z}^{V(T)}$
be two functions.
For a vertex $u$ of $T$,
and a function $\rho\in \mathbb{Z}^{V(T)}$,
let $V_u$ be the subset of $V(T)$ containing $u$ and its descendants,
let $T_u$ be the subtree of $T$ induced by $V_u$,
and let $\rho^{\to u}\in \mathbb{Z}^{V(T)}$ 
be the function with
$$\rho^{\to u}(v)=
\begin{cases}
\rho(v) & \mbox{, if $v\in V(T)\setminus \{ u\}$, and}\\
\rho(v)-1 & \mbox{, if $v=u$.}
\end{cases}$$
Below we consider threshold functions of the form 
$\rho|_{V_u}+\rho'|_{V_u}$ for the subtrees $T_u$, 
where $\rho$ and $\rho'$ are defined on sets containing $V_u$. For notational simplicity,
we omit the restriction to $V_u$
and write `$\rho+\rho'$' instead of `$\rho|_{V_u}+\rho'|_{V_u}$' 
in these cases.
For an integer $k$ and a non-negative integer $b$, 
let $[k]$ be the set of positive integers at most $k$, 
and let 
$${\cal P}_k(b)=\left\{ (b_1,\ldots,b_k)\in \mathbb{N}_0^k:b_1+\cdots+b_k=b\right\}$$
be the set of ordered partitions of $b$ into $k$ non-negative integers.

Our approach to show Theorem \ref{theorem1} is similar as in \cite{ehra} and
relies on recursive expressions for the following two quantities:
For a vertex $u$ of $T$ and a non-negative integer $b$, let
\begin{itemize}
\item $x_0(u,b)$ be the maximum of ${\rm dyn}(T_u,\tau+\iota)$
over all $\iota\in \mathbb{Z}^{V_u}$ with 
$0\leq \iota(v)\leq \iota_{\max}(v)$ for every $v\in V_u$ , and
$\iota(V_u)=b$, and
\item $x_1(u,b)$ be the maximum of ${\rm dyn}\left(T_u,(\tau+\iota)^{\to u}\right)$
over all $\iota\in \mathbb{Z}^{V_u}$ with 
$0\leq \iota(v)\leq \iota_{\max}(v)$ for every $v\in V_u$ , and
$\iota(V_u)=b$.
\end{itemize}
The increment $\iota$ captures the local increases of the thresholds within $V_u$.
The value $x_1(u,b)$ corresponds to a situation,
where the infection reaches the parent of $u$ before it reaches $u$,
that is, the index $0$ or $1$ indicates the amount of help
that $u$ receives from outside of $V_u$.

Note that $x_j(u,b)=-\infty$ 
if and only if $b>\iota_{\max}(V_u)$
for both $j$ in $\{ 0,1\}$.
If $b\leq\iota_{\max}(V_u)$,
then let $\iota_0(u,b),\iota_1(u,b)\in \mathbb{Z}^{V_u}$ with 
$0\leq \iota_j(u,b)\leq \iota_{\max}$, and
$\iota_j(u,b)(V_u)=b$ for both $j\in  \{ 0,1\}$,
be such that 
\begin{eqnarray*}
x_0(u,b)&=&{\rm dyn}\Big(T_u,\tau+\iota_0(u,b)\Big)\mbox{ and}\\
x_1(u,b)&=&{\rm dyn}\Big(T_u,\Big(\tau+\iota_1(u,b)\Big)^{\to u}\Big),
\end{eqnarray*}
where, if possible, let $\iota_0(u,b)=\iota_1(u,b)$.
As we show in Corollary \ref{corollary1} below,
$\iota_0(u,b)$ always equals $\iota_1(u,b)$,
which is a key fact for our approach.

\begin{lemma}\label{lemma0}
	$x_0(u,b)\geq x_1(u,b)$, and if $x_0(u,b)=x_1(u,b)$, then $\iota_0(u,b)=\iota_1(u,b)$.
\end{lemma}
\begin{proof}
If $x_1(u,b)=-\infty$, then the statement is trivial.
Hence, we may assume that $x_1(u,b)>-\infty$,
which implies that the function $\iota_1(u,b)$ is defined.
Let $D$ be a minimum dynamic monopoly of $\left(T_u,\tau+\iota_1(u,b)\right)$.
By the definition of $x_0(u,b)$, we have $x_0(u,b)\geq |D|$.
Since $D$ is a dynamic monopoly of $\left(T_u,(\tau+\iota_1(u,b))^{\to u}\right)$,
we obtain 
$x_0(u,b)\geq |D|\geq {\rm dyn}\left(T_u,(\tau+\iota_1(u,b))^{\to u}\right)=x_1(u,b)$.
Furthermore, if $x_0(u,b)=x_1(u,b)$, then 
$x_0(u,b)=|D|={\rm dyn}\left(T_u,\tau+\iota_1(u,b)\right)$,
which implies $\iota_0(u,b)=\iota_1(u,b)$.
\end{proof}

\begin{lemma}\label{lemma1}
If $u$ is a leaf of $T$, 
and $b$ is an integer with $0\leq b\leq \iota_{\max}(u)$, 
then, for $j\in \{ 0,1\}$,
\begin{eqnarray*}
x_j(u,b)&=&
\begin{cases}
0 &\mbox{, if $\tau(u)+b-j\leq 0$,}\\
1 &\mbox{, otherwise, and}
\end{cases}\\
\iota_j(u,b)(u)&=&b.
\end{eqnarray*}
\end{lemma}
\begin{proof}
These equalities follow immediately from the definitions.
\end{proof}

\begin{lemma}\label{lemma2}
Let $u$ be a vertex of $T$ that is not a leaf, 
and let $b$ be a non-negative integer.
If $v_1,\ldots,v_k$ are the children of $u$,
and ${\iota_0}(v_i,b_i)={\iota_1}(v_i,b_i)$ 
for every $i\in [k]$ and every integer $b_i$
with $0\leq b_i\leq \iota_{\max}(V_{v_i})$,
then, for $j\in \{ 0,1\}$,
\begin{eqnarray}
x_j(u,b) &=& z_j(u,b)\mbox{, and}\label{e1}\\
{\iota_0}(u,b)&=&{\iota_1}(u,b),\mbox{ if $b\leq \iota_{\max}(V_u)$},\label{e3}
\end{eqnarray}
where $z_j(u,b)$ is defined as
$$
\max\left\{ \delta_j(b_u,b_1,\ldots,b_k)+\sum\limits_{i=1}^kx_1(v_i,b_i):(b_u,b_1,\ldots,b_k)\in {\cal P}_{k+1}(b)\mbox{ with }b_u\leq \iota_{\max}(u)\right\},
$$
and, for $(b_u,b_1,\ldots,b_k)\in {\cal P}_{k+1}(b)$
 with $b_u\leq \iota_{\max}(u)$,
\begin{eqnarray*}
\delta_j(b_u,b_1,\ldots,b_k)& :=&
\begin{cases}
0 &\mbox{, if $\Big|\Big\{ i\in [k]:x_0(v_i,b_i)=x_1(v_i,b_i)\Big\}\Big|\geq \tau(u)+b_u-j$, and}\\
1 &\mbox{, otherwise.}
\end{cases}
\end{eqnarray*}
\end{lemma}
\begin{proof}	
By symmetry, it suffices to consider the case $j=0$.

First, suppose that $b>\iota_{\max}(V_u)$.
If $(b_u,b_1,\ldots,b_k)\in {\cal P}_{k+1}(b)$ with 
$b_u\leq \iota_{\max}(u)$,
then $b_i>\iota_{\max}(V_{v_i})$ for some $i\in [k]$,
which implies $z_0(u,b)=-\infty=x_0(u,b)$.

Now, let $b\leq n(T_u)$, which implies $x_0(u,b)>-\infty$.
The following two claims complete the proof of (\ref{e1}).
\begin{claim}\label{claim1}
	$x_0(u,b)\geq z_0(u,b)$.
\end{claim}
\begin{proof}[Proof of Claim \ref{claim1}]
It suffices to show that
$x_0(u,b)\geq \delta_0(b_u,b_1,\ldots,b_k)+\sum\limits_{i=1}^kx_1(v_i,b_i)$
for every choice of $(b_u,b_1,\ldots,b_k)$ in ${\cal P}_{k+1}(b)$ with $b_u\leq \iota_{\max}(u)$
and $b_i\leq \iota_{\max}(V_{v_i})$ for every $i\in [k]$.
Let $(b_u,b_1,\ldots,b_k)$ be one such an element.
Let $\iota_u\in \mathbb{Z}^{V_u}$ be defined as 
\begin{eqnarray}\label{eiotau}
	\iota_u(v)=
	\begin{cases}
		b_u &\mbox{, if $v=u$, and}\\
		0 &\mbox{, otherwise,}
	\end{cases}
\end{eqnarray}
and let $\iota=\iota_u+\sum\limits_{i=1}^k\iota_1(v_i,b_i)$,
where $\iota_1(v_i,b_i)(u)$ is set to $0$ for every $i\in [k]$.
Since $\iota(V_u)=b$
and $0\leq\iota\leq\iota_{\max}$, we have $x_0(u,b)\geq {\rm dyn}(T_u,\tau+\iota)$.

Let $D$ be a minimum dynamic monopoly of $(T_u,\tau+\iota)$,
that is, $|D|\leq x_0(u,b)$.
For each $i\in [k]$,
it follows that the set $D_i=D\cap V_{v_i}$ 
is a dynamic monopoly of $\left(T_{v_i},(\tau+\iota)^{\to v_i}\right)$.
Since, restricted to $V_{v_i}$, 
the two functions $(\tau+\iota)^{\to v_i}$
and $(\tau+\iota_1(v_i,b_i))^{\to v_i}$ coincide,
we obtain 
$$|D_i|\geq {\rm dyn}\Big(T_{v_i},\Big(\tau+\iota_1(v_i,b_i)\Big)^{\to v_i}\Big)\geq x_1(v_i,b_i).$$
If $\delta_0(b_u,b_1,\ldots,b_k)=0$,
then $|D|\geq \sum\limits_{i=1}^k|D_i|\geq \delta_0(b_u,b_1,\ldots,b_k)+\sum\limits_{i=1}^kx_1(v_i,b_i)$.
Similarly, if $u\in D$, then 
$|D|=1+\sum\limits_{i=1}^k|D_i|
\geq \delta_0(b_u,b_1,\ldots,b_k)+\sum\limits_{i=1}^kx_1(v_i,b_i)$.
Therefore, we may assume that 
$\delta_0(b_u,b_1,\ldots,b_k)=1$ and that $u\not\in D$.
This implies that there is some $\ell\in [k]$ with $x_0(v_\ell,b_\ell)>x_1(v_\ell,b_\ell)$
such that $D_\ell=D\cap V_{v_\ell}$ is a dynamic monopoly of $\left(T_{v_\ell},\tau+\iota\right)$.
Since, by assumption, 
$\iota_0(v_\ell,b_\ell)=\iota_1(v_\ell,b_\ell)$, we obtain that, restricted to $V_{v_\ell}$,
the two functions $\tau+\iota$ and $\tau+\iota_0(v_\ell,b_\ell)$ coincide,
which implies
$|D_\ell|\geq 
{\rm dyn}\left( T_{v_\ell},\tau+\iota_0(v_\ell,b_\ell)\right)
=x_0(v_\ell,b_\ell)\geq 1+x_1(v_\ell,b_\ell)$.
Therefore, also in this case,
$|D|=|D_\ell|+\sum\limits_{i\in [k]\setminus \{\ell\}}|D_i|
\geq \delta_0(b_u,b_1,\ldots,b_k)+\sum\limits_{i=1}^kx_1(v_i,b_i)$.
\end{proof}

\begin{claim}\label{claim3}
	$x_0(u,b)\leq z_0(u,b)$.
\end{claim}
\begin{proof}[Proof of Claim \ref{claim3}]
Let $\iota=\iota_0(u,b)$, that is, $x_0(u,b)={\rm dyn}(T_u,\tau+\iota)$.
Let $b_i=\iota(V_{v_i})$ for every $i\in [k]$, and let $b_u=b-\sum\limits_{i=1}^kb_i$.
Clearly, $(b_u,b_1,\ldots,b_k)\in{\cal P}_{k+1}(b)$ and $b_u\leq \iota_{\max}(u)$.
Let $D_i$ be a minimum dynamic monopoly of 
$\left(T_{v_i},(\tau+\iota)^{\to v_i}\right)$
for every $i\in [k]$.
By the definition of $x_1(v_i,b_i)$, 
we obtain $|D_i|\leq x_1(v_i,b_i)$.
Let $D=\{ u\}\cup \bigcup\limits_{i=1}^kD_i$.
The set $D$ is a dynamic monopoly of $(T_u,\tau+\iota)$, which implies $x_0(u,b)\leq |D|$.

If $\delta_0(b_u,b_1,\ldots,b_k)=1$,
then 
$$x_0(u,b)
\leq |D|
=1+\sum\limits_{i=1}^k|D_i|
\leq \delta_0(b_u,b_1,\ldots,b_k)+\sum\limits_{i=1}^kx_1(v_i,b_i)
\leq z_0(u,b).$$
Therefore, we may assume that $\delta_0(b_u,b_1,\ldots,b_k)=0$.
By symmetry, we may assume that 
$x_0(v_i,b_i)=x_1(v_i,b_i)$ for every $i\in [\tau(u)+b_u]$.
Let $D_i'$ be a minimum dynamic monopoly of 
$\left(T_{v_i},\tau+\iota\right)$ for every $i\in [\tau(u)+b_u]$.
By the definition of $x_0(v_i,b_i)$, 
we obtain $|D'_i|\leq x_0(v_i,b_i)=x_1(v_i,b_i)$.
Let $D'=\bigcup\limits_{i\in [\tau(u)+b_u]}D'_i
\cup \bigcup\limits_{i\in [k]\setminus [\tau(u)+b_u]}D_i$.
The set $D'$ is a dynamic monopoly of $(T_u,\tau+\iota)$.
This implies
$$x_0(u,b)\leq |D'|
=\sum\limits_{i\in [\tau(u)+b_u]}|D'_i|
+\sum\limits_{i\in [k]\setminus [\tau(u)+b_u]}|D_i|
\leq\sum\limits_{i\in [k]}x_1(v_i,b_i)
\leq z_0(u,b),$$
which completes the proof of the claim.
\end{proof}
It remains to show (\ref{e3}).
If $x_0(u,b)=x_1(u,b)$, then (\ref{e3}) follows from Lemma \ref{lemma0}. Hence, we may assume that $x_0(u,b)>x_1(u,b)$.
Since, by definition, 
$$\delta_1(b_u,b_1,\ldots,b_k)\leq \delta_0(b_u,b_1,\ldots,b_k)\leq \delta_1(b_u,b_1,\ldots,b_k)+1$$ 
for every $(b_u,b_1,\ldots,b_k)\in {\cal P}_{k+1}(b)$ with $b_u\leq\iota_{\max}(u)$,
we obtain $z_1(u,b)\leq z_0(u,b)\leq z_1(u,b)+1$.
Together with (\ref{e1}), 
the inequality $x_0(u,b)>x_1(u,b)$ implies that 
\begin{eqnarray*}
	x_0(u,b)&=&z_0(u,b)>z_1(u,b)=x_1(u,b)\mbox{ and}\\ z_1(u,b)&=&z_0(u,b)-1.
\end{eqnarray*}
Let $(b_u,b_1,\ldots,b_k)\in {\cal P}_{k+1}(b)$ with $b_u\leq\iota_{\max}(u)$ be such that 
$$z_0(u,b)=
\delta_0(b_u,b_1,\ldots,b_k)+\sum\limits_{i=1}^kx_1(v_i,b_i).$$
We obtain
\begin{eqnarray*}
	z_1(u,b)
	&\geq &
	\delta_1(b_u,b_1,\ldots,b_k)+\sum\limits_{i=1}^kx_1(v_i,b_i)\\
	&\geq &\delta_0(b_u,b_1,\ldots,b_k)-1+\sum\limits_{i=1}^kx_1(v_i,b_i)\\
	&=&z_0(u,b)-1\\
	&=&z_1(u,b),
\end{eqnarray*}
which implies 
$z_1(u,b)=
\delta_1(b_u,b_1,\ldots,b_k)+\sum\limits_{i=1}^kx_1(v_i,b_i)$,
that is, the same choice of $(b_u,b_1,\ldots,b_k)$ in ${\cal P}_{k+1}(b)$
with $b_u\leq\iota_{\max}(u)$ maximizes the terms defining $z_0(u,b)$ and $z_1(u,b)$.

Since $z_0(u,b)>z_1(u,b)$,
we obtain
$\delta_1(b_u,b_1,\ldots,b_k)=0$
and 
$\delta_0(b_u,b_1,\ldots,b_k)=1$,
which,
by the definition of $\delta_j$,
implies that there are exactly $\tau(u)+b_u-1$ indices $i$ in $[k]$
with $x_0(v_i,b_i)=x_1(v_i,b_i)$.
By symmetry, we may assume that 
$x_0(v_i,b_i)=x_1(v_i,b_i)$ for $i\in [\tau(u)+b_u-1]$
and 
$x_0(v_i,b_i)>x_1(v_i,b_i)$ for $i\in [k]\setminus 
[\tau(u)+b_u-1]$.

Let $\iota=\iota_u+\sum\limits_{i=1}^k\iota_0(v_i,b_i)$,
where 
$\iota_0(v_i,b_i)(u)$ is set to $0$ for every $i\in [k]$
and $\iota_u$ is as in~(\ref{eiotau}).
Note that, by assumption, we have $\iota=\iota_u+\sum\limits_{i=1}^k\iota_1(v_i,b_i)$.
Let $D$ be a minimum dynamic monopoly of $(T_u,\tau+\iota)$.
By the definition of $x_0(u,b)$, we have $|D|\leq x_0(u,b)$.
Let $D_i=D\cap V_{v_i}$ for every $i\in [k]$.
Since $D_i$ is a dynamic monopoly of $\left(T_{v_i},(\tau+\iota)^{\to v_i}\right)$ for every $i\in [k]$, we obtain $|D_i|\geq x_1(v_i,b_i)$.
Note that 
\begin{itemize}
\item either $u\in D$,
\item or $u\not\in D$ and there is some index $\ell\in [k]\setminus [\tau(u)+b_u-1]$ such that $D_\ell=D\cap V_{v_\ell}$ is a dynamic monopoly of $(T_{v_\ell},\tau+\iota)$.
\end{itemize} 
In the first case,
we obtain 
$$z_0(u,b)=x_0(u,b)\geq |D|
=1+\sum\limits_{i=1}^k|D_i|
\geq 1+\sum\limits_{i=1}^kx_1(v_i,b_i)
=z_0(u,b),$$
and, in the second case, 
we obtain $|D_\ell|\geq x_0(v_\ell,b_\ell)\geq x_1(v_\ell,b_\ell)+1$, and, hence,
$$z_0(u,b)=x_0(u,b)\geq |D|
=|D_\ell|+\sum\limits_{i\in [k]\setminus \{ \ell\}}|D_i|
\geq 1+\sum\limits_{i=1}^kx_1(v_i,b_i)
=z_0(u,b).$$
In both cases we obtain $|D|=x_0(u,b)$,
which implies that $\iota_0(u,b)$ may be chosen equal to $\iota$.

Now, let $D^-$ be a minimum dynamic monopoly of $\left(T_u,(\tau+\iota)^{\to u}\right)$.
By the definition of $x_1(u,b)$, we have $|D^-|\leq x_1(u,b)$.
Let $D^-_i=D^-\cap V_{v_i}$ for every $i\in [k]$.
Since $D^-_i$ is a dynamic monopoly of $\left(T_{v_i},(\tau+\iota)^{\to v_i}\right)$ for every $i\in [k]$, we obtain $|D^-_i|\geq x_1(v_i,b_i)$.
Now,
$$z_1(u,b)=x_1(u,b)
\geq |D^-|
\geq \sum\limits_{i=1}^kx_1(v_i,b_i)
=z_1(u,b),$$
which implies that $|D^-|=x_1(u,b)$,
and that $\iota_1(u,b)$ may be chosen equal to $\iota$.
Altogether, 
the two functions $\iota_0(u,b)$ and $\iota_1(u,b)$ may be chosen equal, 
which implies (\ref{e3}).
\end{proof}
Applying induction 
using Lemma \ref{lemma1} and Lemma \ref{lemma2}, 
we obtain the following.

\begin{corollary}\label{corollary1}
$\iota_0(u,b)=\iota_1(u,b)$ for every vertex $u$ of $T$, 
and every integer $b$ with $0\leq b\leq \iota_{\max}(V_u)$.
\end{corollary}
Apart from the specific values of $x_0(u,b)$ and $x_1(u,b)$,
the arguments in the proof of Lemma~\ref{lemma2} 
also yield feasible recursive choices for $\iota_0(u,b)$.
In fact, 
if 
$$x_0(u,b)
=\delta_0(b_u,b_1,\ldots,b_k)+\sum\limits_{i=1}^kx_1(v_i,b_i)>-\infty$$
for $(b_u,b_1,\ldots,b_k)\in {\cal P}_{k+1}(b)$ with $b_u\leq \iota_{\max}(u)$, 
and $\iota_u$ is as in (\ref{eiotau}),
then 
$\iota_u+\sum\limits_{i=1}^k\iota_0(v_i,b_i)$
is a feasible choice for $\iota_0(u,b)$.

Our next lemma explains how to efficiently 
compute the expressions in Lemma \ref{lemma2}.

\begin{lemma}\label{lemma3}
Let $u$ be a vertex of $T$ that is not a leaf, 
let $b$ be an integer with $0\leq b\leq \iota_{\max}(V_u)$,
and let $v_1,\ldots,v_k$ be the children of $u$.
If the values $x_1(v_i,b_i)$ are given for every $i\in [k]$ 
and every integer $b_i$ with $0\leq b_i\leq \iota_{\max}(V_{v_i})$,
then $x_0(u,b)$ and $x_1(u,b)$ can be computed in time $O\left(k^2(b+1)^2\right)$.
\end{lemma}
\begin{proof}
By symmetry, it suffices to explain how to compute $z_0(u,b)$.

For $p\in \{0\}\cup [k]$, 
an integer $p_=$, 
an integer $b'\in \{0\}\cup [b]$,
and $b_u\in\{0\}\cup[\min\{\iota_{\max}(u), b'\}]$, 
let $M(p,p_=,b',b_u)$ be defined as the maximum of the expression
$\sum\limits_{i=1}^px_1(v_i,b_i)$
over all 
$(b_1,\ldots,b_p)\in {\cal P}_{p}(b'-b_u)$
such that $p_=$ equals $\Big|\Big\{ i\in [p]:x_0(v_i,b_i)=x_1(v_i,b_i)\Big\}\Big|.$
Clearly, $M(p,p_=,b',b_u)=-\infty$ if 
$p<p_=$ 
or $p_=<0$ 
or $b'-b_u>\sum\limits_{i=1}^p \iota_{\max}(V_{v_i})$, 
and 
\begin{eqnarray*}
M(0,0,b',b_u)=
\begin{cases}
0 &\mbox{, if $b'=b_u$, and}\\
-\infty &\mbox{, otherwise.}
\end{cases}	
\end{eqnarray*}
For $p\in [k]$, the value of $M(p,p_=,b',b_u)$ 
is the maximum of the following two values:
\begin{itemize}
	\item The maximum of 
	$M(p-1,p_=-1,b_{\leq p-1},b_u)+x_1(v_p,b_p)$
	over all $(b_{\leq p-1},b_p)\in {\cal P}_2(b'-b_u)$
	with $x_0(v_p,b_p)=x_1(v_p,b_p)$, and
	\item the maximum of 
	$M(p-1,p_=,b_{\leq p-1},b_u)+x_1(v_p,b_p)$
	over all $(b_{\leq p-1},b_p)\in {\cal P}_2(b'-b_u)$
	with $x_0(v_p,b_p)>x_1(v_p,b_p)$,
\end{itemize}
which implies that $M(p,p_=,b',b_u)$ can be determined in $O(b'+1)$ time
given the values 
$$\mbox{$M(p-1,p_=,b_{\leq p-1},b_u)$, $M(p-1,p_=-1,b_{\leq p-1},b_u)$, $x_0(v_p,b_p)$, and $x_1(v_p,b_p)$.}$$
Altogether, the values $M(k,p_=,b,b_u)$ 
for all $p_=\in \{ 0\}\cup [k]$ can be determined in time $O\left(k^2(b+1)\right)$.

For $b_u\in\{0\}\cup[\min\{\iota_{\max}(u), b\}]$, let $m(b_u)$ be the maximum of the two expressions
$$1+\max\Big\{ M(k,p_=,b,b_u):p_=\in \{ 0\}\cup [\tau(u)-b_u-1]\Big\}$$
and
$$\max\Big\{ M(k,p_=,b,b_u):p_=\in [k]\setminus [\tau(u)-b_u-1]\Big\}.$$
Now, by the definition of $\delta_0(b_u,b_1,\ldots,b_k)$,
the value of $z_0(u,b)$ equals
$\max\Big\{ m(b_u) : b_u\in\{0\}\cup[\min\{\iota_{\max}(u), b\}]
\Big\}.$
Hence, $z_0(u,b)$ can be computed in time $O\left(k^2(b+1)^2\right)$.
\end{proof}
We proceed to the proof of our first theorem.

\begin{proof}[Proof of Theorem \ref{theorem1}]
Given $(T,\tau,\iota_{\max},b)$,
Lemma \ref{lemma1} to Lemma \ref{lemma3} imply
that the values of $x_0(u,b')$ and of $x_1(u,b')$ 
for all $u\in V(T)$ and all $b'\in \{ 0\}\cup [b]$
can be determined in time
$$O\left(\sum\limits_{u\in V(T)}d_T(u)^2(b+1)^2\right).$$
It is a simple folklore exercise that $\sum\limits_{u\in V(T)}d_T(u)^2\leq n^2-n$
for every tree $T$ of order $n$, 
which implies the statement about the running time.
Since ${\rm vacc}(T,\tau,\iota_{\max},b)=x_0(r,b)$, 
the statement about the value of ${\rm vacc}(T,\tau,\iota_{\max},b)$ follows.
The statement about the increment $\iota$
follows easily from the remark after Corollary \ref{corollary1} 
concerning the function $\iota_0(u,b)$, and the proof of Lemma~\ref{lemma3},
where, next to the values $M(p,p_=,b',b_u)$, 
one may also memorize suitable increments.
\end{proof}
We conclude with the proof of our second theorem.

\begin{proof}[Proof of Theorem \ref{theorem2}]
Let $G$ be an $r$-regular graph of order $n$,
and let $b$ be an integer with $(2r-1)(r+1)\leq b\leq rn=2m(G)$.

Let 
$\iota\in \mathbb{Z}^{V(G)}$
with $0\leq \iota\leq d_G$
and $\iota(V(G))=b$ be such that 
${\rm vacc}(G,0,d_G,b)={\rm dyn}(G,\iota)$.
By a result of Ackerman et al. \cite{acbewo},
$${\rm vacc}(G,0,d_G,b)={\rm dyn}(G,\iota)
\leq \sum\limits_{u\in V(G)}\frac{\iota(u)}{d_G(u)+1}
=\frac{\iota(V(G))}{r+1}=\frac{b}{r+1}.$$
First, suppose that the matching number $\nu$ of $G$
satisfies $2r\nu>b$.
In this case, 
$G$ has a matching $M$ with 
$\tau_M(V(G))=2r|M|\leq b$
and $2r(|M|+1)\geq b+1$,
where $\tau_M$ is as in the statement.
We obtain
$2{\rm dyn}(G,\tau_M)\geq 2|M|
\geq 2\left(\frac{b+1}{2r}-1\right)
\geq \frac{b}{r+1}
\geq {\rm vacc}(G,0,d_G,b)$.
Next, suppose that $2r\nu\leq b$.
If $M$ is a maximum matching and $D$ is a minimum vertex cover,
then $|D|\leq 2|M|$.
Since $D$ is a dynamic monopoly of $(G,d_G)$,
we obtain 
$2{\rm dyn}(G,\tau_M)
\geq 2|M|
\geq |D|
\geq {\rm dyn}(G,d_G)
\geq {\rm vacc}(G,0,d_G,b)$,
that is,
$2{\rm dyn}(G,\tau_M)\geq {\rm vacc}(G,0,d_G,b)$
holds in both cases.
\end{proof}

\end{document}